%% file: principal6.tex
\documentclass[a4paper,11pt]{article}
\usepackage[utf8x]{inputenc}
\usepackage[title]{appendix}
\usepackage{varioref}
\title{A simple proof of the Wirsching-Goodwin representation of integers connected to 1 in the $3x+1$ problem.}

\author{
  Daudin, Jean-Jacques\\
  \texttt{jeanjacques.daudin@gmail.com}
  \and
  Pierre, Laurent\\
  \texttt{laurent.pierre@u-paris10.fr}
}

\input{localdefs}

\usepackage{amsthm}
\newtheorem{theorem}{Theorem}
\newtheorem{definition}{Definition}
\newtheorem{lemma}{Lemma}
\newtheorem{proposition}{Proposition}

\newtheorem*{summary}{Summary}
\newtheorem*{corollary}{Corollary}
\def\Z{{Z\!\!\!Z}} 

\begin{document}

\maketitle

\begin{summary}
This paper gives a simple proof of the Wirsching-Goodwin representation of integers connected to 1 in the $3x+1$ problem (see \cite{Wirsching} and \cite{Goodwin}). This representation permits to compute all the ascending Collatz sequences  $(f^{(i)}(n),\: i=1,b-1)$ with a last value $f^{(b)}(n)=1.$ Other periodic sequences connected to $1$ are also identified.
\end{summary}

\section{Basic elements}
In the presentation of the book "The Ultimate Challenge: The 3x+1 Problem", \cite{Lagarias}, J.C. Lagarias write {\it The $3x+1$ problem, or Collatz problem, concerns the following seemingly innocent arithmetic procedure applied to integers: If an integer $x$ is odd then "multiply by three and add one", while if it is even then "divide by two". The $3x+1$ problem asks whether, starting from any positive integer, repeating this procedure over and over will eventually reach the number 1. Despite its simple appearance, this problem is unsolved.} We refer to this book and other papers from the same author for a good review of the context and the references.

\subsection{Definitions}
Let $n \in \mathbb{N}$.
\subsubsection*{Direct algorithm}

$$ T(n)=\left\lbrace \begin{array}{ll}
3n+1 & if \quad n\equiv 1 \pmod2 \\
n/2 & if \quad n\equiv 0 \pmod2 \\
\end{array}
 \right.
$$
\subsubsection*{Inverse algorithm}
$$ U(n)=\left\lbrace 
2n \quad  and \quad \dfrac{n-1}{3} \quad   if \quad n\equiv 4 \pmod6 
\right\rbrace
$$

\subsubsection*{Graph $G(n)$}
Let $(n_1,n_2) \in \mathbb{N}^2.$ $n_1$ and $n_2$ are connected by an edge if $n_1=T(n_2)$ or $n_2=T(n_1)$. $G(n)$ is the subset of the integers connected to n.

\subsubsection*{Conjecture "$3x+1$"}
$\forall n \in \mathbb{N}, \exists k \in \mathbb{N} :T^k(n)=1.$ 
An equivalent assertion is $G(1)=\mathbb{N}^*$.

\section{Restriction to odd integers}

\subsection{ $f$ and  $h$} \label{def f et h}

If the "$3x+1$" conjecture is true for the odd integers it is also true for the even ones by definition of $T$. The expressions of $T$ and $U$ restricted to odd terms are the following with $n$ odd:
\begin{itemize}
\item $T$ becomes $f$: $f(n)=(3n+1)2^{-j(3n+1)}$ with $j(3n+1)$ the power of 2 in the prime factors decomposition of $3n+1$. $f$ is often called the "Syracuse function".
\item $U$ becomes $h$, see\cite{Colussi}:
$$ h(n)=\left\lbrace \begin{array}{ll}
\emptyset & {\rm if} \quad n\equiv 0 \pmod3 \\
{ \frac{n2^k-1}3, k=2,4,6...} & {\rm if} \quad n\equiv 1 \pmod3 \\
{ \frac{n2^k-1}3, k=1,3,5...} & {\rm if} \quad n\equiv 2 \pmod3 \\
\end{array}
 \right.
$$
\end{itemize}

The expression of $h$ comes from the following:

\begin{equation} \label{h}
f(n)=(3n+1)2^{-j(3n+1)} \Rightarrow n=\frac{f(n)2^{j(3n+1)}-1}3 \in \mathbb{N}
\end{equation}
There are 3 cases
\begin{itemize}
\item $f(n)\equiv0 \pmod 3$ (\vref{h}) is impossible, 
\item  $f(n)\equiv 1 \pmod 3 \Rightarrow f(n)=3x+1 $ with $x \in \mathbb{N}$
\begin{eqnarray*}
\frac{f(n)2^{j(3n+1)}-1}3 & = & \frac{(3x+1)2^{j(3n+1)}-1}3 \\
& = & x.2^{j(3n+1)}+\frac{2^{j(3n+1)}-1}3 \\
& \in & \mathbb{N} \quad {\rm if} \quad j(3n+1) \; {\rm even} \quad \mbox{(see Lemma\vref{lem2power2k})}
\end{eqnarray*}
 
\item  $f(n)\equiv 2 \pmod 3 \Rightarrow f(n)=3x+2$ with $x \in \mathbb{N}$
\begin{eqnarray*}
\frac{f(n)2^{j(3n+1)}-1}3 & = & \frac{(3x+2)2^{j(3n+1)}-1}3 \\
& = & x.2^{j(3n+1)}+\frac{2^{j(3n+1)+1}-1}3 \\
& \in & \mathbb{N} \quad {\rm if} \quad j(3n+1) \; {\rm odd} \quad \mbox{(see Lemma\vref{lem2power2k})}
\end{eqnarray*}
\end{itemize}
\subsection{Graph $g(n)$}

Let $(n_1,n_2)$ be odd integers. $n_1$ and $n_2$ are connected by an edge if $n_1=f(n_2)$ or $n_1=f(n_2)$. $g(n)$ is the subset of the odd integers connected to n.

\section{Properties of $g(1)$}
Many authors have tried to give a representation of $g(1).$ Goodwin \cite{Goodwin} gives a short history of their work
and provides some more results. This paper follows the same line of research and
the theorems \ref{theoremFractions}, \ref{theoremGoodwin1} and \ref{theoremGoodwin2} are not new.
However we tried to give a simple and clear presentation of the results and the proofs.
The theorems \ref{theoremv_1} and \ref{theoremdesuns} are new as long as we know.  

\subsection{Expression of $n \in g(1)$ as a sum of fractions}

\begin{theorem} \label{theoremFractions}
Let $n \in g(1). \; \exists  (b,a>u_1>u_2,...>u_b=0) \in \mathbb{N}^{b+2} :\;$
$$ n=\frac{2^a}{3^b}-\sum_{i=1,b}\frac{2^{u_i}}{3^{b-i+1}}.$$
\end{theorem}

 Note that $\frac{2^a}{3^b} \geq 1 \Rightarrow a \geq b\frac{log3}{log2}.$
 \begin{proof}
 $ n\in g(1) \Leftrightarrow \exists b : n \in h^{(b)}(1).$ The proof uses induction with $b$. Theorem 1 is true for $b=1$ because $h(1)=\left\lbrace \frac{2^k-1}{3}, k=2,4,6..., \right\rbrace,$ and for $b=2$ because
 \begin{eqnarray*}
  h^{(2)}(1) & \subset &  \left\lbrace \frac{1}{3}\left( \frac{2^{k_1}-1}{3}2^{k_2}-1\right), k_1=2,4,6...;k_2 \in \mathbb{N} \right\rbrace \\
  & \subset &  \left\lbrace  \frac{2^{k_1+k_2}}{3^2}-\frac{2^{k_2}}{3^2}-\frac{2^0}{3}, k_1=2,4,6...;k_2 \in \mathbb{N} \right\rbrace 
 \end{eqnarray*}
Assume that theorem 1 is true for $l \leq b-1$.
\begin{eqnarray*}
  h^{(b)}(1) & \subset &  \left\lbrace  \frac{1}{3}\left[ \left( \frac{2^a}{3^{b-1}}-\sum_{i=1,b-1}\frac{2^{u_i}}{3^{b-i}}\right)2^{k}-1\right] , (a>u1>...u_{b-1}=0,k) \in \mathbb{N}^{b+1} \right \rbrace\\
  & \subset & \left \lbrace  \frac{2^{a+k}}{3^b}-\sum_{i=1,b-1}\frac{2^{u_i+k}}{3^{b-i+1}}-\frac{2^0}{3},(a>u1>...u_{b-1}=0,k) \in \mathbb{N}^{b+1} \right \rbrace
 \end{eqnarray*}
 The last expression has the form claimed in theorem 1.
 \end{proof}
 Note that $k_1$ is even but $k_l$ may be odd or even for $l>1$. Thus $a-u_1$ is even. If  $a-u_1=2$, $h(1)=1$, so the first "interesting" value is $a-u_1=4$.

\begin{proof}
An alternative proof of theorem \vref{theoremFractions} using $f$ :
$n \in g(1) \Leftrightarrow \exists b \in \mathbb{N} :f^b(n)=1.$ $b$ is the number of odd integers (excluding 1) in the sequence from $n$ to 1. 

Induction with $b$:

Let $b=1$ and $3n+1=2^{j(3n+1)}x$ be a partial prime factors decomposition of $3n+1.$  
\begin{eqnarray*}
f(n) & = & (3n+1)2^{-j(3n+1)} \\
     & = & 2^{j(3n+1)}x2^{-j(3n+1)} \\
      & = & x  
\end{eqnarray*}
$$b=1 \Rightarrow f(n)=1 \Rightarrow x=1 \Rightarrow 3n+1=2^{j(3n+1)}\Rightarrow n=\frac{2^{j(3n+1)}}{3}-\frac{1}{3}.$$

Let $b=2$.
$b=2 \Rightarrow f(f(n))=1 \Rightarrow  f(n)=\frac{2^{j(3f(n)+1)}}{3}-\frac{1}{3}.$

$$f(n)=(3n+1)2^{-j(3n+1)}\Rightarrow (3n+1)2^{-j(3n+1)}=\frac{2^{j(3f(n)+1)}}{3}-\frac{1}{3}.$$

Therefore
$$ n=\frac{2^{j(3n+1)+j(3f(n)+1}}{3^2}-\frac{2^{j(3n+1)}}{3^2}-\frac{1}{3}.$$

Assuming that the theorem is true till $b-1$ we have to prove that it is true for $b.$

$$f(n)  =  (3n+1)2^{-j(3n+1)}= \frac{2^a}{3^{b-1}}-\sum_{i=1,b-1}\frac{2^{u_i}}{3^{b-i}}.$$

Therefore
$$ n=\frac{2^{a+j(3n+1)}}{3^b}-\sum_{i=1,b-1}\frac{2^{u_i+j(3n+1)}}{3^{b-i+1}}-\frac{1}{3}.$$ 
\end{proof}
 Note that the general form of $u_i$ is thus $u_{b-i}=\sum_{l=1,i}j[3f^{(l-1)}(n)+1]$ , with $f^{(0)}=Id$,  and $a=\sum_{l=1,b}j[3f^{(l-1)}(n)+1]$. 
 
\subsection{Admissible tuple $(b,a>u_1>u_2,...>u_b=0)$}
Only some values of $(b,a>u_1>u_2,...>u_b=0)$ give an integer $n$ in theorem 1, most of them do not. 

\begin{definition}
A tuple $(b,a \geq b\frac{log3}{log2}, a>u_1>u_2,...>u_b=0)$ of $b+1$ integers is admissible if $\frac{2^a}{3^b}-\sum_{i=1,b}\frac{2^{u_i}}{3^{b-i+1}} \in \mathbb{N}.$
\end{definition}

The admissible parity of $u_i-u_{i+1}$ is determined by the remainder modulo 3 of the integer obtained at step $i$ (see the definition of $h$ in section \vref{def f et h}).
 
\begin{lemma} \label{lemmePeriod}
Let $\frac{n2^k-1}{3} \in \mathbb{N}$ and $\frac{n2^k-1}{3} \equiv v \pmod 3$. Then $\frac{n2^{k+2}-1}{3} \equiv v+1 \pmod 3$.
\end{lemma}

The lemma indicates that $v$ is a periodic function of $k$ with period 6: $\frac{n2^k-1}{3} \equiv  \frac{n2^{k+6}-1}{3} \pmod 3.$

\begin{proof}
$\frac{n2^k-1}{3} \equiv v \pmod 3 \Rightarrow \frac{n2^k-1}{3}=3x+v,$
\begin{eqnarray*}
\frac{n2^{k+2}-1}{3} & = & 4\frac{n2^{k}-1}{3}+1 \\
& = & 4(3x+v)+1 \\
& \equiv & v+1 \pmod 3
\end{eqnarray*} 
\end{proof}

\begin{lemma} \label{lemmeNegation}
Let $n \in \mathbb{N}$ and $n_1=\frac{n2^k-1}{3} \notin \mathbb{N}$. Then $\forall l \in \mathbb{N}, \quad \frac{n_12^l-1}{3}  \notin \mathbb{N}$.
\end{lemma}

The lemma indicates that if $(b, a>u_1>u_2,...>u_b=0)$ is admissible and $k$ has not the correct parity, the tuple $(b+1,a+k>u_1+k>u_2+k,..>u_b+k,u_{b+1}=0)$ is not admissible and all tuples based on it are also not admissible. Conversely, if $(b, a>u_1>u_2,...>u_b=0)$ is admissible, $(b-1, a-u_{b-1}>u_1-u_{b-1}>u_2-u_{b-1},...>u_{b-1}-u_{b-1}=0)$ is also admissible and all such successive reduced tuples till $(1, a-u_1>u_1-u_1=0)$

\begin{proof}
 $n_1=\frac{n2^k-1}{3}=\frac{p}{3},$ with $p$ and 3 relatively prime. $\frac{n_12^l-1}{3}=\frac{\frac{p}{3}2^l-1}{3}=\frac{p2^l-3}{9}.$
 Suppose that $\frac{p2^l-3}{9}=x  \in \mathbb{N}.$ Then $p2^l=9x+3$ that is impossible because $p$ and 3 are relatively prime.
\end{proof}

\subsection{Structure of $g(1)$}

\begin{lemma}
$g(1)$ is a tree with an additionnal loop in its root $1$. 
\end{lemma}

\begin{proof}
Let $h^*$ be a modified version of $h$:
$ h^*(1)={ \frac{2^k-1}{3}, k=4,6,8...}, $

 $g^*(1)=\left\lbrace 1 \cup h^*(1) \cup h[h^*(1)],..\cup h^{(l)}[h^*(1)]..\right\rbrace.$ 
 The case $n \in h(n_1) \cap h(n_2)$ with $n_1 \neq n_2,$ is impossible because there is only one $f(n).$
 Thus $g^*(1)$ is a tree because any $n \in g^*(1)$ cannot have two different parents. $g(1)$ is equal to $g^*(1)$ with a supplementary loop at node 1.
\end{proof}

The following definition \vref{def g(t,s)} and proposition \vref{prop g(t,s)} are not nessessary for the proof of theorem \vref{theoremGoodwin2} and may be skipped.
\begin{definition} \label{def g(t,s)}
$g^*(1)[t,s] \subset g^*(1)$ is the graph generated by the admissible tuples with $b \leq t,$ $ 4 \leq a-u_1 \leq 2+6s$ and $u_i-u_{i+1} \leq 6s.$ 
\end{definition}

\begin{proposition} \label{prop g(t,s)}
$\vert g^*(1)[t,s] \vert= 1+\frac{3s[(2s)^t-1]}{2s-1},$ with $\vert A \vert$  the cardinal of the set $A$.
\end{proposition}

\begin{proof}
Lemma \vref{lemmePeriod} implies that for each node of the tree there are $3s$ admissible children of which $2s$ have children.
\end{proof}
Note that with $s=1$ one obtains that the ratio of integers pertaining to $g^*(1)[t,1]$ and less than $\max(g^*(1)[t,1])\simeq\frac{2^{2+6t}}{3^t}$ is greater than $\frac{3}{4}\left(\frac{3}{2^5} \right)^t.$

\begin{lemma} \label{lemmeperiod3}
Let $ (b,u_0>u_1>u_2,...>u_b=0) $ be an admissible tuple. Let $j<b$ and $u'_i=u_i+2.3^{b-j-1} \; if \; i \leq j$ and $u'_i=u_i \; if \; i> j.$ Then the tuple $ (b,u'_0>u'_1>u'_2,...>u'_b=0)$ is admissible.
\end{lemma}

\begin{proof}
Let $ n=\frac{2^{u_0}}{3^b}-\sum_{i=1,b}\frac{2^{u_i}}{3^{b-i+1}},$ and  $x=\frac{2^{{u'}_0}}{3^b}-\sum_{i=1,b}\frac{2^{u'_i}}{3^{b-i+1}}.$
\begin{eqnarray*}
x-n & = & \left( \frac{2^{u_0}}{3^b}-\sum_{i=1,j}\frac{2^{u_i}}{3^{b-i+1}} \right)\left(2^{2.3^{b-j-1}}-1 \right) \\
& = & \left( \frac{2^{u_0}}{3^b}-\sum_{i=1,j}\frac{2^{u_i}}{3^{b-i+1}} \right)\left(q3^{b-j} \right) \\
& = & q\left( \frac{2^{u_0}}{3^j}-\sum_{i=1,j}\frac{2^{u_i}}{3^{j-i+1}} \right) 
\end{eqnarray*}
Lemma \vref{LemmePuissancesde3Modulo3} implies that $q \in \mathbb{N}$ and lemma \vref{lemmeNegation} implies that the second term is integer, therefore $x$ is integer.
\end{proof}

We introduce an alternative notation for the tuple $ (b,u_0>u_1>u_2,...>u_b=0) .$

 Let $v_i=u_{i-1}-u_{i}, \; i=1,...b.$ The tuple $ (b,\sum_{i=1,b}v_i, \sum_{i=2,b}v_i,... v_b)$ is equal to the tuple $ (b,u_0>u_1>u_2,...>u_b=0) .$ The alternative notation for this tuple is $(b,v_1,v_2,...,v_b).$ 
 
 Note that $v_i=j(3f^{(b-i)}(n)+1),$ with $n$ given by theorem\vref{theoremFractions}, see
the second proof of theorem\vref{theoremFractions}.

\begin{theorem} \label{theoremGoodwin1}
Let $ v_i \in \mathbb{N}, \; i=2,...b  \; with \; 1 \leq v_i \leq 2.3^{b-i} \; and \; b>1. \;$ For each tuple $(v_2,v_3,...v_b) \; \exists v_1 \; even  \; with \; 4 \leq v_1 \leq 2.3^{b-1}$ such that $ (b,\sum_{i=1,b}v_i, \sum_{i=2,b}v_i,... v_b)$ is admissible.
\end{theorem}
\begin{proof}
The cardinal number of $F=\{v_2,...v_b\}$ is
\begin{eqnarray*}
\vert \{v_2,...v_b\} \vert & = & \prod_{i=2,b}2.3^{b-i} \\
& = & 2^{b-1}3^{\sum_{i=2,b}(b-i)} \\
& = & 2^{b-1}3^{\sum_{k=0,b-2}k} \\
& = & 2^{b-1}3^{\frac{(b-2)(b-1)}{2}} \\
\end{eqnarray*}

Let $E$ be the set of the admissible $\{v_1,v_2,...v_b\}.$ $\#E$ is equal to the product of the number of admissible nodes with children at each step excepted the last one with sterile nodes taken into account. At the first step $v_1,$ this number is $\frac{2}{3}3^{b-1}$. Then for each $v_1$ there are $3^{b-2}$ possible admissible values for $v_2$. From these values only  $\frac{2}{3}3^{b-2}$ have children, and so on till the last step with one admissible node (with or without child for this last step). The product is equal to $2^{b-1}3^{\frac{(b-2)(b-1)}{2}}.$ 

Let $t : E \mapsto F$ with $t(v_1,v_2,...v_b)=(v_2,...v_b)$. $t$ is injective because $t(v'_1,v'_2,...v'_b)=t(v_1,v_2,...v_b) \Rightarrow (v'_2,...v'_b)=(v_2,...v_b).$
$n=\frac{2^a}{3^b}-\sum_{i=1}^{b}\frac{2^{u_i}}{3^{b-i+1}} \in \mathbb{N}$ and
$n'=\frac{2^{a'}}{3^b}-\sum_{i=1}^{b}\frac{2^{u_i}}{3^{b-i+1}} \in \mathbb{N}.$ Therefore 
$n'-n=\frac{2^{a'}}{3^b}-\frac{2^a}{3^b}=2^a\frac{2^{a'-a}-1}{3^b} =2^a\frac{2^{v_1'-v_1}-1}{3^b}\in \mathbb{N}$ and thus 
$v'_1-v_1=p.2.3^{b-1}$ with $p \geq 1$. Therefore $v'_1=v_1$ and $t$ is injective. $\#E=\#F$ and $t$ injective imply that $t$ is bijective and that only one $v_1 \le 2.3^{b-1}$ corresponds to a t-uple $(v_2,...v_b)$.

 Note that the number of admissible $\{v_2,...v_b\}$ corresponding to one $v_1$ is $$\frac{2^{b-1}3^{\frac{(b-2)(b-1)}{2}}}{\frac{2}{3}3^{b-1}} = 2^{b-2}3^{\frac{(b-3)(b-2)}{2}}.$$ 
\end{proof}

\begin{center}
\begin{table}[ht]
\label{Tableb=3}
\begin{center}
\begin{tabular}{|r|r|r|r|r|r|r|r|r|r|r|r|r|}
\hline
v1* & 4 & 4 & 8 & 8 & 10 & 10 & 14 & 14 & 16 & 16 & 20 & 20\\
\hline
v2 & 3 & 5 & 2 & 6 & 1 & 5 & 4 & 6 & 1 & 3 & 2 & 4\\
\hline
v3 & 2 & 1 & 1 & 2 & 1 & 2 & 2 & 1 & 2 & 1 & 2 & 1 \\
\hline
n & 17 & 35 & 75 & 2417 & 151 & 4849 & ... &  &  &  &  & 1242755 \\
\hline
\end{tabular}
\end{center}
\caption{The 12 admissible tuples with $b=3$}
\end{table}
\end{center}

Among the possible sequences $(v_2,...v_b)$ allowed by the Theorem \vref{theoremGoodwin1} some are specially interesting such as the strictly ascending sequence  $(f^{(i)}(n),\: i=1,b-1)$ (see $n=151$ in the above table as an example), given in the following corollary. 
\begin{corollary}
$ \forall b \in \mathbb{N}, \: \exists  n \in g(1) \: : \forall i \in (1,b-1), \:  f^{(i)}(n) > f^{(i-1)}(n).$
\end{corollary}
\begin{proof}
$n$ is obtained with $v_i=1,\: i=2:b,$ and $v_1$ given by theorem \vref{theoremGoodwin1} and lemma \vref{lemmeperiod3}.
\end{proof}
 The Wirsching-Goodwin representation of the nodes of $g(1)$ obtained with $b$ steps (see \cite{Goodwin}) may be now stated in the following theorem. Let $g^*(1,b)=\{ n \in g(1)\; : \; f^{(b)}=1 \; \mbox{and} \; f^{(b-1)} \neq 1\} $ and $v_1^*$ the value of $v_1$ whose existence is proven in theorem \vref{theoremGoodwin1}. 
 \begin{theorem} \label{theoremGoodwin2}
There is a one to one relation between $g^*(1,b)$ with $b>1$ and the set of the  tuples $ (b,v'_1,v'_2,...,v'_b)$ with $v'_i=v_i+2.3^{b-i}c_i,$ $c_i \in \mathbb{N^*}$, $v_i \in \mathbb{N}, \; i=2,...b  \; with \; 1 \leq v_i \leq 2.3^{b-i} $ and $4 \leq v_1=v_1^* \leq 2.3^{b-1}$.
\end{theorem}
\begin{proof}
Direct from theorem\vref{theoremGoodwin1} and lemma \vref{lemmeperiod3}.
\end{proof}

For each $b$ and $(v_2,...v_b)$ there is a unique $v_1 \in  (4, 2.3^{b-1}+2).$  The theorem \ref{theoremv_1} give its value. 

\begin{theorem} \label{theoremv_1}
$v_1^*=a-\sum_{i=2,b}v_i$ with
$$  a \equiv  \log_2\left( \sum_{i=1,b}2^{u_i}3^{i-1} \bmod 3^b \right) \pmod{2.3^{b-1}}.$$
 
\end{theorem}
 
\begin{proof}
$$ n=\frac{2^a}{3^b}-\sum_{i=1,b}\frac{2^{u_i}}{3^{b-i+1}} \Rightarrow 2^a=n3^b+\sum_{i=1,b}2^{u_i}3^{i-1},$$
The group $(\Z/3^b\Z)^*$ is cyclic of order $2.3^{b-1}$ and generated by $2\bmod 3^b$ (see \cite{Vinogradov}).
This means that
\begin{eqnarray*}
F : \Z/2.3^{b-1}\Z & \to & (\Z/3^b\Z)^* \\
 i\bmod{2.3^{b-1}} & \mapsto & 2^i \bmod 3^b.
\end{eqnarray*}
is defined and bijective. So we can use its reciprocal $F^{-1}$ and call it $\log_2$.

E.g. $\log_2(7\bmod 9)=4\bmod 6$ \quad since \quad $2^4=16\equiv 2^{10}=1024\equiv 7\pmod 9$.

$  2^a\equiv \sum_{i=1,b}2^{u_i}3^{i-1} \pmod{3^b}$ implies $$ a \bmod 2.3^{b-1}
= \log_2\left( \sum_{i=1,b}2^{u_i}3^{i-1} \bmod 3^b\right)$$
\end{proof}

\subsection{Ascending Collatz sequences excepted the last term}
It is possible to give explicitely $a$ and $v_1$  in some particular cases such as $(v_2=v_3=...=v_b=1)$ and any $b$. The following theorem defines all the strictly ascending sequence  $(f^{(i)}(n),\: i=1,b-1)$ with a last value $f^{(b)}(n)=1.$ 

\begin{theorem} \label{theoremdesuns}
Let $(v_2=v_3=...=v_b=1)$ then $v_1^*=3^{b-1}+1,$ with the corresponding 
$$n=\frac{2^{b-1}2^{3^{b-1}+1}}{3^b}-\sum_{i=1}^{b}\frac{2^{b-i}}{3^{b-i+1}}.$$
\end{theorem}

\begin{center}
\begin{table}[ht]
\label{Table_f_croissante}
\begin{center}
\begin{tabular}{|r|r|r|r|r|r|r|}
\hline
$v_1^*$ & $v_2$ & $v_3$ & $v_4$ & $v_5$ & $v_6$ & n \\
\hline
4   & 1 &   &   &   &   & 3\\
\hline
10  & 1 & 1 &   &   &   & 151\\
\hline
28  & 1 & 1 & 1 &   &   &  26512143  \\
\hline
82  & 1 & 1 & 1 & 1 &   &  318400215865581346424671 \\
\hline
244 & 1 & 1 & 1 & 1 & 1 &  ...  \\
\hline
\end{tabular}
\end{center}
\caption{ $v_1^*$ for $b=2,...,6$ and $v_i=1, i=2...b$}
\end{table}
\end{center}

\begin{proof}
With induction with $b$. The theorem is true for $b=2$ and $b=3$ (see the above table). Assume that it is true till $b-1$. Let
$n_j=\frac{2^{j-1}2^{3^{b-1}+1}}{3^j}-\sum_{i=1}^{j}\frac{2^{j-i}}{3^{j-i+1}}, \: j=1,b,$
the values obtained at step $j$ with $b$ total steps, and
$m_j=\frac{2^{j-1}2^{3^{b-2}+1}}{3^j}-\sum_{i=1}^{j}\frac{2^{j-i}}{3^{j-i+1}}, \: j=1,b-1,$
the values obtained at step $j$ with $b-1$ total steps. The $m_j$ are integers by the induction hypothesis. Note that
\begin{eqnarray*}
3^{b-1}+1 & = & 3^{b-2}+1 +2.3^{b-2} \\
& = & v_1^*(b-1) +2.3^{b-2}
\end{eqnarray*}
Therefore the lemma \vref{lemmeperiod3} implies that $n_1,...n_{b-1}$ are integers. We have to prove that $n_b \in \mathbb{N}.$ This is true if $n_{b-1} \equiv2 \pmod 3.$
\begin{eqnarray*}
n_{b-1}-m_{b-1} & = & \frac{2^{b-2}2^{3^{b-1}+1}}{3^{b-1}}-\sum_{i=1}^{b-1}\frac{2^{b-1-i}}{3^{b-i}}-\frac{2^{b-2}2^{3^{b-2}+1}}{3^{b-1}}+\sum_{i=1}^{b-1}\frac{2^{b-1-i}}{3^{b-i}} \\
& = & \frac{2^{b-2}2^{3^{b-1}+1}}{3^{b-1}}-\frac{2^{b-2}2^{3^{b-2}+1}}{3^{b-1}}\\
& = & 2^{b-1+3^{b-2}}\frac{ 2^{2.3^{b-2}}-1}{3^{b-1}}\\
& \equiv & 2^{b-1+3^{b-2}} \pmod 3 \hbox{,\quad see lemma \vref{LemmePuissancesde3Modulo3}}\\
& \equiv & (-1)^{b-1+3^{b-2}} \pmod 3\\
& = & (-1)^b \hbox{,\quad since $3^{b-2}-1$ is even}
\end{eqnarray*}  
Using Lemma \vref{LemmeParitb} one obtains
$$ n_{b-1} \equiv m_{b-1}+(-1)^b \equiv \left\lbrace \begin{array}{ll}
1+1 =      2 \pmod 3 & {\rm if} \quad b\equiv 0 \pmod2 \\
0-1 \equiv 2 \pmod 3 & {\rm if} \quad b\equiv 1 \pmod2 \\
\end{array}
 \right.
$$

\end{proof}

A simpler (but not directly related to theorem \ref{theoremGoodwin2}) proof of theorem \ref{theoremdesuns} is the following:
a Collatz sequence $(n_0,n_1,\ldots,n_b)$ is growing if $\forall i,\;n_{i+1}=(3n_i+1)/2$.
Therefore $\exists k\in \mathbb{N}^*$, $n_0=k 2^b-1$ and $n_b=k3^b-1$ and $n_i=k 3^i 2^{b-i}-1$.
If the following term of the sequence is 1, then $\exists l\in \mathbb{N}^*$, $1=(k 3^{b+1}-2)/2^l$.
Thus $k 3^{b+1}=2+2^l$ and $2^l\equiv -2\pmod{3^{b+1}},$
which is equivalent to $l-1\equiv 3^b\pmod{2\times 3^b}$.
The first term of this Collatz sequence is
$n_0=(1+2^{(2p+1)3^b})(2/3)^{b+1}-1$ with $p\in\mathbb{N}$.

Theorem \ref{theoremdesuns} is  generalized by the following proposition.

\begin{proposition} \label{prop}
Let $v_i=k \in \mathbb{N}^*, i=2,b$. $v_1$ is defined by the relation $2^{v_1-k}(2^k-3) \equiv 1 \pmod{3^b}.$ 
\end{proposition}

\begin{proof}
 $u_i=k(b-i),$ therefore
$\sum_{i=1,b}2^{u_i}3^{i-1}=2^{k(b-1)}.\sum_{i=1,b}{(2^k)}^{-i+1}3^{i-1}=\frac{2^{kb}-3^b}{2^k-3}.$

Theorem \ref{theoremv_1} implies that  $2^a \equiv \frac{2^{kb}-3^b}{2^k-3} \pmod{3^b}$

$a=v1+(b-1)k \Rightarrow 2^{v_1+(b-1)k} \equiv \frac{2^{kb}-3^b}{2^k-3} \pmod{3^b}.$

$2^{v_1+(b-1)k}(2^k-3) \equiv 2^{kb} \pmod{3^b} \Rightarrow
2^{v_1-k}(2^k-3) \equiv 1 \pmod{3^b}.$
\end{proof}
If $k=1,2$ one obtains explicitely all the values of $v_1$, but this is not true for $k \ge 3.$

If $k=1,$ 
$ 2^{v_1-1} \equiv  -1 \pmod{3^b} \Rightarrow v_1-1 \equiv 3^{b-1} \pmod{2.3^{b-1}}.$

If $k=2,$ 
 $v_1-2 \equiv 0 \pmod{2.3^{b-1}}.$

If $k=3,$  $5.2^{v_1-3}\equiv 1 \pmod{3^b} \Rightarrow v_1\equiv 3-\log_2(5\bmod 3^b) \pmod{2.3^{b-1}}$.

This kind of result can be extended to any periodic sequence $(v_2,v_3,...)$. For exemple the sequence $v_{2i}=1, v_{2i+1}=2, \; i=1,...(b-1)/2$ implies that $2^{v_1} \equiv -20 \pmod{3^b}$. This result is obtained by dividing $ \sum_{i=1,b}2^{u_i}3^{i-1}$ in two separate geometric series that gives $3^{b+1}-2^{\frac{3(b+1)}{2}}+12\left(3^{b-1}-2^{\frac{3(b-1)}{2}} \right).$ Note that this particular sequence is associated to a globally increasing Collatz sequence till the penultimate term.

\subsection{Structure of $g(n)$}
The structure of $g(n)$ for $n \in \mathbb{N},$ is similar to the structure of $g(1)$ (for $n \in g(1)$ or $n \notin g(1)$). Proofs are very similar to the case of $g(1)$ and are not given here.
Theorem \ref{theoremFractions} is slightly modified:
\begin{theorem} \label{theoremFractionsbis}
Let $m \in g(n), n \equiv (1,2) \pmod 3. \; \exists  (b,a>u_1>u_2,...>u_b=0) \in \mathbb{N}^{b+2} :\;$
$$ m=n.\frac{2^a}{3^b}-\sum_{i=1,b}\frac{2^{u_i}}{3^{b-i+1}}.$$
\end{theorem}
Theorem \ref{theoremGoodwin2} remains true for $g(n)$:
 \begin{theorem} \label{theoremGoodwin2bis}
If $n \equiv (1,2) \pmod 3$, there is a one to one relation between $g(n,b)$ with $b>1$ and the set of the  tuples $ (b,v'_1,v'_2,...,v'_b)$ with $v'_i=v_i+2.3^{b-i}c_i,$ $c_i \in \mathbb{N^*}$, $v_i \in \mathbb{N}, \; i=2,...b  \; with \; 1 \leq v_i \leq 2.3^{b-i} $ and $1 \leq v_1=v_1^* \leq 2.3^{b-1}$.
\end{theorem}
If $n \equiv 2 \pmod 3$, $v_1$ is odd. If $n \equiv 1 \pmod 3$, $v_1$ is even.
The Theorem \ref{theoremv_1} becomes
\begin{theorem} \label{theoremv_1bis}
Let $n \equiv (1,2) \pmod 3$.

 $v_1^*=a-\sum_{i=2,b}v_i$ with
$$  a \equiv  \log_2\left( \sum_{i=1,b}2^{u_i}3^{i-1} \bmod 3^b \right)-\log_2(n \bmod 3^b) \pmod{2.3^{b-1}}.$$
\end{theorem}

Finally the proposition \ref{prop} may be extented to any starting number:

\begin{proposition} \label{prop2}
Let $n \equiv (1,2) \pmod 3$ and $v_i=k \in \mathbb{N}^*, i=2,b$. $v_1$ is defined by the relation $n.2^{v_1-k}(2^k-3) \equiv 1 \pmod{3^b}.$ 
\end{proposition}
This proposition with $k=1$ allows to define all the strictly ascending sequences from $m$ to $n$, $(f^{(i)}(m),\: i=1,b)$ with  $f^{(b)}(m)=n.$ Contrarily to the case of $g(1)$, the sequence $v_1=...v_b=1$ exists and  the proposition \ref{prop2} imply that   $n \equiv -1 \pmod{3^b}.$ $n$ is odd therefore $n=2p.3^b-1,$ with the associate value $m=2^{b+1}p-1.$  Moreover  the proposition \ref{prop2} shows that  $v_1=...v_b=2$ imply that $n \equiv 1 \pmod{3^b}.$ $n$ is odd therefore $n=2p.3^b+1,$  with the associate value $m=2^{2b+1}p+1.$ Note that this includes the case $p=0$ and $n=m=1.$

\bibliographystyle{plain} 
\bibliography{biblio}

\begin{appendix}
\section{Proofs of Lemmas}
\begin{lemma} \label{lem2power2k}
$$ 2^k  \overset{\pmod 3}{\equiv}  \left\lbrace \begin{array}{ll}
 2 & {\rm if} \quad k \quad {\rm odd}   \\
 1 & {\rm if} \quad k \quad {\rm even}  \\
\end{array}
 \right.
$$
\end{lemma}
\begin{proof}
$2\equiv-1\pmod3$. So 
if $k$ is even then $2^k\equiv(-1)^k=1\pmod3$.

If $k$ is odd then $2^k\equiv(-1)^k=-1\equiv2\pmod3$.
%
%
\end{proof}

\begin{lemma} \label{LemmePuissancesde3bis}
$\frac{2^{3^k}+1}{3^{k+1}} \in \mathbb{N} \mbox{ and } \equiv 1 \pmod 3 $
\end{lemma}

\begin{proof}
By induction. The lemma is true for $k=0$.

 Assuming the lemma true for $k-1$ implies that $2^{3^{k-1}}+1=3^kx$ with $x \equiv 1 \pmod 3.$
\begin{eqnarray*}
2^{3^k}+1 & = & \left(2^{3^{k-1}}\right)^3+1 \\
& = & \left(3^kx-1\right)^3+1\\
& = & 3^{3k}x^3-3^{2k+1}x^2+3^{k+1}x\\
\frac{2^{3^k}+1}{3^{k+1}}&=&3^{2k-1}x^3-3^kx^2+x  \equiv x \equiv 1 \pmod 3
\end{eqnarray*}
\end{proof}

\begin{lemma} \label{LemmePuissancesde3Modulo3}
$\frac{2^{2.3^k}-1}{3^{k+1}}\equiv 1 \pmod 3  $
\end{lemma}

\begin{proof}

\begin{eqnarray*}
 \frac{2^{2.3^k}-1}{3^{k+1}} & = &  \frac{2^{3^k}+1}{3^{k+1}}(2^{3^k}-1) \\
   & \equiv &  2^{3^k}-1 \pmod 3 \mbox{,  see Lemma \vref{LemmePuissancesde3bis}}\\
   & \equiv & 1 \pmod 3 \mbox{,  see Lemma \vref{lem2power2k}}
\end{eqnarray*}
 
 \end{proof}

\begin{lemma} \label{LemmeParitb}
Let $n_b$ given by theorem \vref{theoremdesuns}. 
$$ n_b \equiv \left\lbrace \begin{array}{ll}
0 \pmod 3 & {\rm if} \quad b\equiv 0 \pmod2 \\
1 \pmod 3 & {\rm if} \quad b\equiv 1 \pmod2 \\
\end{array}
 \right.
$$
\end{lemma}

\begin{proof}
By induction. The lemma is true for $b=2$ and $b=3$ because $n_2=3$ and $n_3=151.$ Assume that the lemma is true till $b-1.$ 
\begin{eqnarray*}
n_{b}-m_{b-1} & = & \frac{2^{b-1}2^{3^{b-1}+1}}{3^{b}}-\sum_{i=1}^{b}\frac{2^{b-i}}{3^{b-i+1}}-\frac{2^{b-2}2^{3^{b-2}+1}}{3^{b-1}}+\sum_{i=1}^{b-1}\frac{2^{b-1-i}}{3^{b-i}} \\
& = & \frac{2^{b-1}2^{3^{b-1}+1}}{3^{b}}-\frac{2^{b-1}}{3^b}-\frac{2^{b-2}2^{3^{b-2}+1}}{3^{b-1}}\\
& = & 2^{b-1}\left( 2\frac{2^{3^{b-1}}+1}{3^b}-\frac{2^{3^{b-2}}+1}{3^{b-1}} \right)\\
& \equiv & 2^{b-1}(2-1) \pmod3 \hbox{,\quad see lemma\vref{LemmePuissancesde3bis}}
\end{eqnarray*}
$$ n_b \equiv m_{b-1} +2^{b-1} \equiv \left\lbrace \begin{array}{ll}
1+2 \equiv 0 \pmod 3 & {\rm if} \quad b\equiv 0 \pmod2 \\
0+1 = 1 \pmod 3 & {\rm if} \quad b\equiv 1 \pmod2 \\
\end{array}
 \right.
$$
\end{proof}

\end{appendix}
\end{document}

%% file: localdefs.tex
\usepackage{geometry}

\usepackage{xspace}
\usepackage{amsmath}
\usepackage{amsfonts,amsmath,amssymb}
\usepackage{xmpmulti}
\usepackage[T1]{fontenc}
\usepackage{amssymb}
\usepackage{url}

\newcommand\Z{\mathcal{Z}}

\usepackage{array}